\documentclass{amsart}
\usepackage{amsmath, amssymb}
\newtheorem{Lem}{Lemma}
\newtheorem{Def}{Definition}
\newtheorem{Them}{Theorem}
\newtheorem{Cor}{Corollary}

\newcommand{\fq}{\mathbb{F}_q}
\newcommand{\rmv}[1]{}

\begin{document}
\title{Partitions and compositions over finite fields}
\author[Muratovi\'{c}-Ribi\'{c}]{Amela Muratovi\'{c}-Ribi\'{c}}
\address{University of Sarajevo, Department of Mathematics, Zmaja od Bosne 33-35, 71000 Sarajevo, Bosnia and Herzegovina}
\email{amela@pmf.unsa.ba}

\author[Wang]{Qiang Wang}
\thanks{Research is partially supported by NSERC of Canada.}
\address{School of Mathematics and Statistics,
Carleton University, 1125 Colonel By Drive,  Ottawa, Ontario, $K1S$ $5B6$, CANADA}

\email{wang@math.carleton.ca}



\begin{abstract}
In this paper we  find exact formulas for the numbers of partitions
and compositions of an element into $m$ parts over a finite field, 
i.e. we find the number of nonzero solutions of the equation
$x_1+x_2+\cdots +x_m=z$ over a finite field when the order does not
matter and when it does, respectively. We also give an application of
our results in the study of polynomials of prescribed ranges over
finite fields.
\end{abstract}

\maketitle
\section{Introduction}

 Let $n$ and $m$ be positive integers. A \emph{composition} of $n$ is an ordered list of positive integers whose sum is $n$.
 A \emph{$m$-composition} of $n$ is an ordered list of $m$ positive integers ($m$ parts) whose sum is $n$.
 It is well known that there is a bijection between all $m$-compositions of $n$ and $(m-1)$-subsets
 of $[n-1] = \{1, 2, \ldots, n-1\}$ and thus there are ${n-1 \choose m-1}$ $m$-compositions of
 $n$ and $2^{n-1}$ compositions of $n$. Similarly, a \emph{weak composition} of $n$ is an ordered list of non-negative
 integers whose sum is $n$  and a \emph{weak $m$-composition}  of $n$ is an ordered list of $m$ non-negative parts whose
 sum is $n$. Using substitution of variables, we can easily obtain that the number of weak $m$-compositions of $n$
 (i.e., the number of non-negative integer solutions to $x_1+ x_2 + \cdots + x_m = n$) is equal to the number
 of $m$-compositions of $n + m$ (i.e., the number of positive integer solutions to $x_1+ x_2 + \cdots + x_m = n+m$), which is
   ${n+m-1 \choose m-1} = {n+m -1 \choose n}$.
The combinatorial interpretation of ${n+m-1 \choose m-1} = {n+m -1
\choose n}$ is the number of ways in selecting $n$-multisets from a set $M$
with $m$ elements, which is sometimes called $n$-combinations of $M$
with repetitions.  Disregarding the order of the summands, we have
the concepts of partitions
of $n$ into $m$ parts, partitions of $n$ into at most $m$ parts,  and so on. For more details we refer the
reader to \cite{Stanley}.

Let $\fq$ be a finite fields of $q=p^r$ elements. The subset problem
over a subset  $D \subseteq \fq$ is to determine  for a given $z \in
\fq$, if there is a nonempty subset $\{x_1, x_2, \ldots, x_m\}
\subseteq D$ such that $x_1 + x_2  + \cdots +x_m  = z$. This subset
sum problem is known to be $NP$-complete.  In the study of the
subset sum problem over finite fields, Li and Wan \cite{LiWan}
estimated the number, $N(m, b, D) = \# \{ \{x_1, x_2, \ldots, x_m\}
\subseteq  D \mid  x_1 + x_2  + \cdots x_m  = z \}$,  of $m$-subsets of $D \subseteq \fq$ whose sum is $z \in \fq$. In
particular, exact formulas are obtained in cases that $D= \fq$ or
$\fq^*$ or $\fq \setminus \{0, 1\}$.  Similarly, we are interested
in the number $S(m, z, D) = \# \{ (x_1, x_2, \ldots, x_m) \in   D
\times D \times \cdots \times D \mid  x_1 + x_2  + \cdots + x_m = z
\}$, that is, the number of  ordered $m$-tuples  whose sum is $z$
and each coordinate belongs to $D \subseteq \fq$, as well as the
number $M(m, z, D)$ which counts the number of  $m$-multisets of $D \subseteq \fq$ whose sum is $z \in \fq$.   In
particular, when $D = \fq$ or $\fq^*$,  this motivated us to
introduce the following.

\begin{Def}
A \emph{partition} of $z\in \fq $ into $m$ parts is a multiset of
$m$ nonzero elements in $\fq^*$ whose sum is $z$.  The $m$ nonzero
elements are the parts of the partition. We denote by $M(m, z,
\fq^*)$ or $\tilde{P}_m(z)$ the  number of partitions of $z$ into $m$ parts
over $\fq$. Similarly,  we denote by $M(m, z, \fq)$ or
$\hat{P}_m(z)$ the  number of partitions of $z$ into at most $m$
parts over $\fq$ and by $\tilde{P}(z)$ the total number of partitions of $z$
over finite field $\fq$.
\end{Def}

We remark that $N(m, z, \fq^*)$ is the number of partitions of an
element $z$ over finite field $\fq$ such that all summands are
distinct, and  $M(m, z, \fq^*)$ is the number of partitions of an
element $z$ into $m$ parts over finite field $\fq$, dropping the
restriction that all summands are distinct.

We also remark that in the study of polynomials of prescribed ranges
over finite fields \cite{MuWa} there has arisen a need as well for
counting the number $M(m, 0, \fq)$ of partitions of $0$ with at most
$m$ parts over finite field $\fq$, which in turn leads us  to answer
a recent conjecture by G\'{a}cs et al on polynomials of prescribed
ranges over finite fields \cite{gacsetal}.

In this article  we first obtain an exact formula for the number of partitions of an element $z \in \fq$ into $m$ parts over $\fq$.

\begin{Them}\label{partitions}
Let $m$ be a non-negative integer, $\fq$ be a finite field of $q=p^r$ elements with prime $p$,  and $z\in \fq$.  The number of
partitions of $z$ into $m$ parts over $\fq$ is given by
$$\tilde{P}_m(z)=\frac{1}{q}\binom{q+m-2}{m}+D_m(z),$$
where
\begin{eqnarray*}
D_m(z) &= \left\{
\begin{array}{ll}
0, &   \mbox{if} ~m\not\equiv 0 ~(\bmod~p) ~\mbox{and}~ m\not\equiv 1 ~(\bmod~p); \\
\frac{q-1}{q}\binom{q/p - 1+j}{j}, &  \mbox{if}~  m =jp,  j \geq 0,  ~\mbox{and} ~z=0 ;\\
-\frac{q-1}{q}\binom{q/p - 1+j}{j}, &  \mbox{if}~  m = jp+ 1, j \geq 0,  ~\mbox{and} ~z=0 ;\\
-\frac{1}{q}\binom{q/p-1+j}{j}, &  \mbox{if}~  m =jp,  j \geq 0,  ~\mbox{and}  ~z \in \fq^*;\\
\frac{1}{q}\binom{q/p-1+j}{j}, &  \mbox{if}~  m = jp+ 1,  j \geq 0,  ~\mbox{and} ~z \in \fq^*.
\end{array}
\right.
\end{eqnarray*}
\end{Them}

Similarly, we have the following definition of compositions over
finite fields.

\begin{Def}
A \emph{composition} of $z\in \fq $ with $m$ parts is a solution
$(x_1, x_2, \ldots, x_m)$  to the equation
\begin{equation} \label{eqCompo}
z=x_1+x_2+\cdots +x_m,
\end{equation}
with each $x_i \in \fq^*$. Similarly, a \emph{weak composition} of $z\in \fq $ with $m$ parts is
a solution $(x_1, x_2, \ldots, x_m)$ to Equation~(\ref{eqCompo}) with each $x_i \in \fq$. 
We denote the number of compositions of $z$ having $m$ parts by
$S(m, z, \fq^*)$ or $S_m(z)$. The number of weak compositions of $z$
with $m$ parts is denoted by $S(m, z, \fq)$.  The total number of
compositions of $z$ over $\fq$ is denoted by  $S(z)$.
\end{Def}

The number of solutions to Equation~(\ref{eqCompo}) (or in more general way as diagonal equations) have been extensively studied. However, it is less studied the number of solutions  (compositions) such that none of variables are zero.
A formula for the number of compositions over $\mathbb{F}_p$  is
given in \cite{BEW}. However, we are not aware of a general formula
for arbitrary $q$ so we also include such a formula here for the
sake of completeness.

\begin{Them}\label{compositions}
Let $m>2$, $\fq$ be a finite field of $q=p^r$ elements with prime $p$, and $z\in \fq$.  The number of compositions of $z$ with
$m$ parts over $\fq$ is given by
$$S_m(z)=(q-1)^{m-2}(q-2)+S_{m-2}(z).$$
It follows that
$$S_m(0)=\frac{(q-1)^m+(-1)^m(q-1)}{q}$$
and
$$S_m(z)=\frac{(q-1)^m-(-1)^m}{q}, ~~~\mbox{if}~~ z\neq 0.$$
\end{Them}

Using the fact that additive group $(\fq,+)$ is isomorphic to the
additive group $(\mathbb{Z}_p^r,+)$,
we obtain that the numbers of partitions and compositions of
elements over $\mathbb{Z}_p^r$ are the same as the numbers of partitions
and compositions of corresponding elements over $\fq$.

Finally, we demonstrate an application of Theorem~\ref{partitions} in the study of polynomials of prescribed range. First let us recall  that the {\it range} of the  polynomial $f(x) \in \fq[x]$ is  a multiset $M$ of size $q$  such that $M = \{ f(x)
: x \in \fq \}$ as a multiset (that is, not only values, but also
multiplicities need to be the same).  Here and also in the following sections we abuse the set notation  for multisets as well.   A nice
reveal of  connections among a combinatorial number theoretical
result, polynomials of prescribed ranges and hyperplanes in vector
spaces over finite fields can be found in \cite{gacsetal}, which we refer
it to the readers  for more details.  In this paper, we obtain
the following result as an application of Theorem~\ref{partitions}.  

\begin{Them} \label{main}
Let $\fq$ be a finite field of $q =p^r$ elements. For every $\ell$
with $\frac{q}{2}\leq \ell < q-3$ there exists a mutiset $M$ with
$\sum_{b\in M} b =0$ and the highest multiplicity $\ell$ achieved at
$0\in M$ such that every polynomial over the finite field $\fq$ with
the prescribed range $M$ has degree greater than $\ell $.
\end{Them}

We note that Theorem~\ref{main} generalizes Theorem 1 in \cite{MuWa} which disproves Conjecture 5.1 in \cite{gacsetal}.  In the following sections, we give the proofs of Theorems 1-3  respectively.

\section{Proof of Theorem~\ref{partitions}}

In this section we prove Theorem~\ref{partitions}. First of all we
prove a few technical lemmas.
%We will show two Lemmas:
\begin{Lem}\label{lemma 1}  Let $a\in
\mathbb{F}_q^*$ and  $m$  be a positive integer. Then
$\tilde{P}_m(a)=\tilde{P}_m(1)$.
\end{Lem}
\begin{proof} Let $x_1+x_2+\cdots +x_m=1$. The following mapping between two multisets defined by $$\{x_1,x_2,\ldots
,x_m\}\mapsto \{ax_1,ax_2,\ldots ,ax_m\}$$ for some $a\in \fq^*$ is
one-to-one and onto, which results in $ax_1+ax_2+\ldots +ax_m=a$.
Thus $\tilde{P}_m(a)=\tilde{P}_m(1)$.
\end{proof}

It is obvious to see that $\tilde{P}_1(z)=1$ if $z\in \fq^*$ and
$\tilde{P}_1(0)=0$. However, we can show that
$\tilde{P}_m(0)=\tilde{P}_m(z)$ if $m\not\equiv 0~(\bmod ~p)$ and
$m\not\equiv 1 ~(\bmod~ p)$ as follows.

\begin{Lem} \label{equality}
Let  $m$  be any positive integer satisfying $m\not\equiv 0~(\bmod ~p)$
and $m\not\equiv 1 ~(\bmod~ p)$.  Then
$\tilde{P}_m(0)=\tilde{P}_m(1)$.
\end{Lem}
\begin{proof}  Let  $x_1+x_2+\cdots +x_m=0$ be a partition of $0$ into $m$ parts. Then $(x_1+1)+(x_2+1)+\cdots +(x_m+1)=m$ is a  partition of $m\in \fq^*$ with at most $m$ parts (if $x_j=p-1$ then $x_j+1=0$), but since $x_j\neq 0$ there is no $x_j+1=1$.  Moreover, there is a bijective correspondence of multisets $\{x_1,\ldots ,x_m\}\mapsto \{x_1+1,\ldots ,x_m+1\}$. Therefore, in order to find the number $\tilde{P}_m(0)$ of partitions of $0$ into $m$ parts over $\fq$,  we need to find the number of partitions of $m$ with at most $m$ parts  but no element is equal to $1$.  This means these partitions of $m$ can have parts equal to the zero.

Let $x_1+x_2+\cdots +x_m=m$. We  assume that the parts equal to 1
(if any) appear in the beginning of the list: $x_1,x_2,\ldots,x_m$.
If $x_1=1$ then $x_1+ x_2 + \cdots +x_m=m$ implies $x_2+\cdots +x_m=m-1$.
Conversely, each partition of $m$ into $m-1$ parts can generate a
partition of $m$ into $m$ parts with the first part equal to $1$. So
the number of partitions of $m$ into $m$ parts with at least one
part equal to 1 is equal to the number of partitions of $m-1$ into
$m-1$ parts.  Let $U_0$ be the family of partitions of $m$ into $m$
parts without zero elements and no part is equal to 
$1$.  Therefore $|U_0|=\tilde{P}_m(m)-\tilde{P}_{m-1}(m-1)$.

\par Let $U_1$ be the family of partitions of $m$ with $m$ parts with
exactly one element equal to $0$ and no element equal to $1$.  Let
$x_1+x_2+\cdots +x_m=m$ be a partition in $U_1$ and  $x_1=0$ and
$x_j \neq 0, 1$ for  $j=2, \ldots, m$. Obviously,  it is equivalent to  a
partition  $x_2+\cdots +x_m=m$ of $m$ into $m-1$ parts with all
parts not equal to $1$. Similarly as in the case for $U_0$ we have
$|U_1|=\tilde{P}_{m-1}(m)-\tilde{P}_{m-2}(m-1)$.

More generally, let $U_i$ be the family of partitions with $m$ parts
with $i$ parts equal to the zero, say $x_1=x_2=\ldots =x_i=0$, and
$x_{j} \neq 0, 1$ for $j=i+1, \ldots, m$. Then we have a partition of
$m$ into $m-i$ parts,  $x_{i+1}+\cdots +x_m=m$, such that no part is
equal to $1$. Similarly, we have
$|U_i|=\tilde{P}_{m-i}(m)-\tilde{P}_{m-i-1}(m-1)$. In particular,
for $i=m-1$ there is only one solution of the equation $x_m=m$ and
thus $|U_{m-1}|= \tilde{P}_1(m)=1$.

We note that these families of $U_i$'s are pairwise disjoint and  their
union is the family  of partitions of $m$  into $m$ parts with no
part equal to $1$. Therefore we have
$\tilde{P}_m(0)=|U_0|+|U_1|+\cdots +|U_{m-1}|=
(\tilde{P}_m(m)-\tilde{P}_{m-1}(m-1))+(\tilde{P}_{m-1}(m)-\tilde{P}_{m-2}(m-1))+\cdots
+(\tilde{P}_2(m)-\tilde{P}_1(m-1))+\tilde{P}_1(m)$.

If  $m\not\equiv 0~(\bmod ~p)$ and $m\not\equiv 1 ~(\bmod~ p)$, then
$m-1$ and $m$ are both nonzero elements in $\fq$. By
Lemma~\ref{lemma 1}, we can cancel
$\tilde{P}_{i}(m-1)=\tilde{P}_{i}(m)$ for $i=1, \ldots, m-1$. Hence
$\tilde{P}_m(0) = \tilde{P}_m(m) = \tilde{P}_m(1)$.
\end{proof}

Using the above two lemmas, we obtain the exact counts of
$\tilde{P}_m(z)$ when  $m\not\equiv 0~(\bmod ~p)$ and $m\not\equiv 1
~(\bmod~ p)$.

\begin{Lem}\label{equality1} If $z\in \fq $ and $m$ is any positive integer satisfying  $m\not\equiv 0~(\bmod ~p)$ and
$m\not\equiv 1 ~(\bmod~ p)$ then  we have
$$\tilde{P}_m(z) = \frac{1}{q}\binom{q+m-2}{m}.$$
\end{Lem}
\begin{proof}
We note that there are  $\binom{(q-1)+m-1}{m}$ multisets of $m$
nonzero elements from $\fq $ in total and the sum of elements in
each multiset can be any element in $\fq$. Using Lemmas~\ref{lemma
1} and \ref{equality} we have $$\sum_{s\in
\fq}\tilde{P}_m(s)=q\tilde{P}_m(1)=\binom{(q-1)+m-1}{m}$$ and
therefore
$$\tilde{P}_m(z)=\tilde{P}_m(1)=\frac{1}{q}\binom{q+m-2}{m}$$
for every $z\in \fq $.
\end{proof}

In order to consider other cases, we use an interesting result
by Li and Wan \cite{LiWan},  which gives the number $N(k,b,\fq^*)$ of sets with
(all distinct) $k$ nonzero elements that sums to  $b\in \fq$. Namely,
\begin{equation}\label{N}
N(k,b,\fq^*)=\frac{1}{q}\binom{q-1}{k}+(-1)^{k+\lfloor
k/p\rfloor}\frac{\nu (b)}{q}\binom{q/p-1}{\lfloor k/p \rfloor},
\end{equation}
where $\nu(b)=-1$ if $b\neq 0$ and $\nu(b)=q-1$ if $b=0$ (see
Theorem 1.2 in \cite{LiWan}).

First we can prove
\begin{Lem}\label{inclusion-exclusion}
Let $N(k,b,\fq^*)$ be  the number of sets with $k$ nonzero elements
that sums to $b\in \fq$ and $m > 1$ be a positive integer.
Then
\begin{eqnarray*}
\tilde{P}_m(0) &=& \left((q-1)N(1,1, \fq^*)\tilde{P}_{m-1}(1)+N(1,0,\fq^*)\tilde{P}_{m-1}(0) \right) \\
&&-\left((q-1)N(2,1,\fq^*)\tilde{P}_{m-2}(1)+N(2,0,\fq^*)\tilde{P}_{m-2}(0)\right)\\
&&+\ldots \\
&&+(-1)^{m-1} \left((q-1)N(m-2, 1,\fq^*)\tilde{P}_2(1)+N(m-2, 0, \fq^*)\tilde{P}_{2}(0)\right) \\
&&+(-1)^m  (q-1)N(m-1,1,\fq^*) +(-1)^{m+1}N(m,0,\fq^*).
\end{eqnarray*}
\end{Lem}
\begin{proof}
Denote by $\mathcal{U}$  the family of all multisets of $m$ nonzero
elements that sums to the zero, i.e. $\tilde{P}_m(0)=|\mathcal{U}|$.
Let $\mathcal{B}_a$ be the family of all multisets of $m$  nonzero
elements
 such that $a$ is a member of each multiset and the sum of elements of each multiset equal to $0$. Namely, $B_a\in \mathcal{B}_a$ implies $\sum_{s\in B_a}s=0$ and $a\in B_a$.  Obviously,  $\mathcal{U}=\bigcup_{a\in \fq^*}\mathcal{B}_a$.

\par Now we will
use the principle of inclusion-exclusion to find the cardinality of
$\mathcal{U}$. For distinct $a_1, \ldots, a_k \in \fq^*$ and $k >
m$, it is easy to see that
$$\mathcal{B}_{a_1}\cap\mathcal{B}_{a_2}\cap\ldots
\cap\mathcal{B}_{a_k}=\emptyset, $$ because each multiset $B_{a_1}$
contains only $m$ nonzero elements.  Moreover, if $k=m$ then  the
number of multisets in the union of intersections is $N(m,0,\fq^*)$.

 If $B\in
\mathcal{B}_{a_1}\cap\mathcal{B}_{a_2}\cap\ldots
\cap\mathcal{B}_{a_k}$ and $k \leq m -1$ then
$$B=\{a_1,a_2,\ldots ,a_k, x_{k+1},\ldots ,x_m\}.$$
Because $x_{k+1}+\cdots + x_m=-(a_1+\cdots +a_k)$, the number of
elements in the intersection
$\mathcal{B}_{a_1}\cap\mathcal{B}_{a_2}\cap\ldots
\cap\mathcal{B}_{a_k}$ is the  same as the number of partitions of
$-(a_1+\cdots +a_k)$ into $m-k$  parts, i.e.
$$|\mathcal{B}_{a_1}\cap\mathcal{B}_{a_2}\cap\ldots
\cap\mathcal{B}_{a_k}|=\tilde{P}_{m-k}(-a_1-\cdots -a_k).$$ We note
that  none of $a_i$'s ($i=1, \ldots, k$) is equal to zero and $N(k,
b,\fq^*) = N(k, 1, \fq^*)$ for any $b\in \fq^*$. In particular,  if
$k < m-1$, then  the sum $a_1 + \cdots + a_{m-1}$ can be any element
in $\fq$ and thus  there are
$(q-1)N(k,1,\fq^*)\tilde{P}_{m-k}(1)+N(k,0,\fq^*)\tilde{P}_{m-k}(0)$
such multisets  $B\in
\mathcal{B}_{a_1}\cap\mathcal{B}_{a_2}\cap\ldots
\cap\mathcal{B}_{a_k}$ for all choices of nonzero distinct $a_1, \ldots,
a_k$.

If $k=m-1$ then the sum $a_1 + \cdots + a_{m-1}$ can not be equal to
the zero, there are in total $(q-1)N(m-1,1,\fq^*)$ such
multisets contained in the intersection of $m-1$ families of
$\mathcal{B}_{a_i}$'s.

Finally we combine the above cases and use the principle of
inclusion-exclusion to complete the proof.
\end{proof}

In the sequel we also need the following result.
\begin{Lem} \label{Mc Claren}
For all  positive integers $s$, we have 
$$\sum_{j=1}^s(-1)^{j+1}\binom{q-1}{j}\binom{q-2+s-j}{s-j}=\binom{q-2+s}{s}$$
\end{Lem}
\begin{proof}  Multiplying
$(1+x)^{q-1}=\sum_{k=0}^{q-1}\binom{q-1}{j}x^j$ and 
series
$$\frac{1}{(1+x)^{q-1}}=\sum_{k=0}^\infty \binom{q-2+k}{k}(-1)^k
x^k,$$ We obtain 
$$1=(1+x)^{q-1}\cdot\frac{1}{(1+x)^{q-1}}=\Bigl(\sum_{k=0}^{q-1}\binom{q-1}{j}x^j\Bigr)\Bigl(\sum_{k=0}^\infty \binom{q-2+k}{k}(-1)^k
x^k\Bigr)=$$
$$\sum_{s=0}^\infty
\Bigl(\sum_{j=0}^s(-1)^{s-j}\binom{q-1}{j}\binom{q-2+s-j}{s-j}\Bigr)x^s.$$
Therefore for $s\geq 1$ we have
$\sum_{j=0}^s(-1)^{s-j}\binom{q-1}{j}\binom{q-2+s-j}{s-j}=0$. This
implies
$$\sum_{j=1}^s(-1)^{s-j+1}\binom{q-1}{j}\binom{q-2+s-j}{s-j}=(-1)^s\binom{q-2+s}{s}.$$
Finally multiplying both sides of the last equality by $(-1)^s$  we
complete the proof.
\end{proof}
\par Next we prove Theorem~\ref{partitions}. In order to do so,
we let
\begin{equation} \label{relation}
\tilde{P}_m(z)=\frac{1}{q}\binom{q-2+m}{m}+D_m(z).
\end{equation}

We assume $q>2$. Obviously,  by Lemma~\ref{equality1}, we have
 $D_m(z)=0$ for any $z\in \fq$ if $m\not\equiv 0~(\bmod ~p)$ and
$m\not\equiv 1~(\bmod ~p)$. Further $D_m(z)=D_m(1)$ by
Lemma~\ref{lemma 1} for all $z\neq 0$.  Because
$\tilde{P}_m(0)+(q-1)\tilde{P}_m(1)=\binom{q-2+m}{m}$,  we have
\begin{equation}\label{Dequality}
D_m(0)+(q-1)D_m(1)=0, ~\text{i.e.,}\quad
D_m(1)=-\frac{1}{q-1}D_m(0).\end{equation}

Next we use convention that $\tilde{P}_0(0)=1$ and $\tilde{P}_0(1)=0$ so that $D_0(0)=\frac{q-1}{q}$ and $D_0(1)=-\frac{1}{q}$. Similarly,   $\tilde{P}_1(0)=0$ and  $\tilde{P}_1(1)=1$
and thus $D_1(0)=-\frac{q-1}{q}$ and $D_1(1)=\frac{1}{q}$. For the
rest of this section, we only need to compute  $D_m(0)$ when $m= jp$
or $m=jp+1$ for some positive integer $j$ because of Equation~(\ref{Dequality}). To do this, we apply
Lemmas~\ref{inclusion-exclusion} and \ref{Mc Claren}, along with
Equations~(\ref{N}) (\ref{relation}), and the following equation
\begin{equation}\label{Nsum}
(q-1)N(m,1,\fq^*)+N(m,0,\fq^*)=\binom{q-1}{m}.\end{equation}

Let us consider  $m=up$ first. In this case, by
Lemma~\ref{inclusion-exclusion} and Equation~(\ref{relation}), we have:
\begin{eqnarray*}
\tilde{P}_m(0)&=&\sum_{s=1}^{m-2}(-1)^{s+1}\left[\frac{1}{q}\binom{q-2+m-s}{m-s}\left((q-1)N(s,1,\fq^*)
+N(s, 0,\fq^*)\right) \right.\\
&& \left. +(q-1)N(s, 1,\fq^*)D_{m-s}(1)+N(s, 0, \fq^*)D_{m-s}(0) \right]\\
&&+(-1)^m (q-1)N(m-1,1,\fq^*)+(-1)^{m+1}N(m,0,\fq^*).
\end{eqnarray*}

Using Equations~(\ref{Nsum}) and (\ref{N}), we obtain
\begin{eqnarray*}
\tilde{P}_m(0)&=& \frac{1}{q}\sum_{s=1}^{m-2}(-1)^{s+1}\binom{q-1}{s}\binom{q-2+m-s}{m-s} \\
&&+\sum_{s=1}^{m-2}(-1)^{s+1}\frac{1}{q}\binom{q-1}{s} \left((q-1)D_{m-s}(1)+D_{m-s}(0)\right)\\
&&+ \sum_{s=1}^{m-2}(-1)^{s+1} (q-1)(-1)^{s+\lfloor s/p\rfloor
}\frac{1}{q}\binom{q/p-1}{\lfloor
s/p\rfloor} \left(-D_{m-s}(1)+D_{m-s}(0) \right) \\
&&+(-1)^m(q-1)\frac{1}{q}\binom{q-1}{m-1}+(-1)^{m+1}\frac{1}{q}\binom{q-1}{m}\\
&&+(-1)^m(q-1)(-1)^{m-1+\lfloor
(m-1)/p\rfloor}\frac{-1}{q}\binom{q/p-1}{\lfloor
(m-1)/p\rfloor}\\
&&+(-1)^{m+1}(-1)^{m+\lfloor m/p\rfloor}\frac{q-1}{q}
\binom{q/p-1}{\lfloor m/p\rfloor}
\end{eqnarray*}

After rearranging terms,  we  use Lemma~\ref{Mc Claren}, Lemma~\ref{equality1},    Equations~(\ref{relation}) and (\ref{Dequality}) to simplify the above as
follows:
\begin{eqnarray*}
&=& \frac{1}{q}\sum_{s=1}^{m}(-1)^{s+1}\binom{q-1}{s}\binom{q-2+m-s}{m-s}\\
&&+ \sum_{1\leq s\leq m-2 \atop s\equiv 0,1
(\bmod{p})}(-1)^{s+1} (-1)^{s+\lfloor s/p\rfloor
}\binom{q/p-1}{\lfloor s/p\rfloor}D_{m-s}(0)\\
&&+(-1)^{u-1}\frac{q-1}{q}\bigl[\binom{q/p-1}{u-1}+\binom{q/p-1}{u}\bigr]\\
&=&\frac{1}{q}\binom{q-2+up}{up} +\sum_{1\leq s\leq up \atop
s\equiv 0,1(\bmod{p})}(-1)^{1+\lfloor s/p
\rfloor }\binom{q/p-1}{\lfloor s/p\rfloor}D_{up-s}(0),\\
\end{eqnarray*}
 where we use Lemma~\ref{Mc Claren} and  $-D_0(0)=D_1(0)=-\frac{q-1}{q}$ to obtain the last equality.  Now let us rewrite this as
\begin{equation}
\label{eq1} \tilde{P}_{up}(0)=\frac{1}{q}\binom{q-2+up}{up}
+\sum_{t=0}^{u-1}(-1)^{1+(u-t)}\binom{q/p-1}{u-t}D_{tp}(0)
+\sum_{t=0}^{u-1}(-1)^{(u-t)}\binom{q/p-1}{u-t-1}D_{tp+1}(0).
\end{equation}

Similarly, for $m=up+1$, we have
\begin{eqnarray*}
&&\tilde{P}_{up+1}(0)\\
&=&\frac{1}{q}\binom{q-2+(up+1)}{up+1}+ \sum_{1\leq s\leq up-1 \atop
s\equiv 0,1(\bmod{p})}(-1)^{1+\lfloor s/p
\rfloor }\binom{q/p-1}{\lfloor s/p\rfloor}D_{up+1-s}(0) \\
&=&
\frac{1}{q}\binom{q-2+(up+1)}{up+1}+\sum_{t=1}^{u-1}(-1)^{1+u-t}\binom{q/p-1}{u-t}
\left( D_{tp}(0) + D_{tp+1}(0) \right) -D_{up}(0).
\end{eqnarray*}

Next we show $D_{up+1}(0)=-D_{up}(0)$ for all $u \geq 0$ by the
 mathematical induction. The base case $u=0$ holds because $D_1(0) = D_0(0) = -\frac{q-1}{q}$.  Assume now $-D_{sp}(0)=D_{sp+1}(0)$ for all $0\leq s< u$ and plug
into  the above formula  we obtain
$$\tilde{P}_{up+1}(0)=\frac{1}{q}\binom{q-2+up+1}{up+1}-D_{up}(0)$$
Because $\tilde{P}_{up+1}(0)=\frac{1}{q}\binom{q-2+up+1}{up+1}+
D_{up+1}(0)$, we conclude that $D_{up+1}(0)=-D_{up}(0)$. Hence it is true for all $u \geq 0$.
Using this relation we simplify Equation~(\ref{eq1}) to
\begin{eqnarray}
\tilde{P}_{up}(0) &=&
\frac{1}{q}\binom{q-2+up}{up}+\sum_{t=0}^{u-1}(-1)^{u-t+1}
\left(\binom{q/p-1}{u-t}+\binom{q/p-1}{u-t-1}\right)D_{tp}(0) \nonumber \\
&=&
\frac{1}{q}\binom{q-2+up}{up}+\sum_{t=0}^{u-1}(-1)^{u-t+1}\binom{q/p}{u-t}D_{tp}(0)
\end{eqnarray}
and by using
$\tilde{P}_{up}(0)=\frac{1}{q}\binom{q-2+up}{up}+D_{up}(0)$ we
obtain
\begin{equation}\label{recursionD}D_{up}(0)=\sum_{t=0}^{u-1}(-1)^{u-t+1}\binom{q/p}{u-t}D_{tp}(0).\end{equation}
Let $f(x)=\sum_{j=0}^\infty D_{jp}(0)x^j$ be the generating function
of the sequence $\{D_{up}(0) : u =0,1 , 2, \ldots\}$. Then
\begin{eqnarray*}
(1-x)^{q/p}f(x)&=&\left(\sum_{l=0}^{q/p}\binom{q/p}{l}(-1)^lx^l
\right) \left(\sum_{j=0}^\infty
D_{jp}(0)x^j \right)\\
&=&D_0(0)+\sum_{u=1}^\infty
\left( \left(\sum_{t=0}^{u-1}\binom{q/p}{u-t}(-1)^{u-t}D_{tp}(0) \right)+D_{up}(0) \right)x^u \\
&=& D_0(0)+\sum_{u=1}^\infty
(-D_{up}(0)+D_{up}(0))x^u=D_0(0)=\frac{q-1}{q}.
\end{eqnarray*}
Now $(1-x)^{q/p}f(x)=\frac{q-1}{q}$ implies
$$f(x)=\frac{q-1}{q}\frac{1}{(1-x)^{q/p}}=\frac{q-1}{q}\sum_{t=0}^\infty
\binom{q/p-1+t}{t}x^t.$$

Hence  $D_{jp}(0)=\frac{q-1}{q}\binom{q/p-1+j}{j}$ for
$j=0,1,2\ldots$. Moreover, we use Equation~(\ref{Dequality}) and $D_{jp+1}(0) = -D_{jp}(0)$ to conclude
$$D_{jp}(0)=\frac{q-1}{q}\binom{q/p-1+j}{j};\qquad
D_{jp}(1)=-\frac{1}{q}\binom{q/p-1+j}{j};$$
$$D_{jp+1}(0)=-\frac{q-1}{q}\binom{q/p-1+j}{j};\qquad
D_{jp+1}(1)=\frac{1}{q}\binom{q/p-1+j}{j}.$$  Finally, together with
Lemma~\ref{equality1} we complete the proof of Theorem~\ref{partitions}.

Finally we note that it is straightforward to derive the following corollary.

\begin{Cor}
Let $m$ be a non-negative integer,  $\fq$ be a finite field of $q=p^r$ elements with prime $p$, and $z\in \fq$.  The number of
partitions of $z$ into at most $m$ parts over $\fq$ is given by
$$\hat{P}_m(z)=\sum_{k=0}^m\tilde{P}_k(z)=\frac{1}{q}\binom{q-1+m}{m}+\tilde{D}_m(z),$$
where
\begin{eqnarray*}
\tilde{D}_m(z) &= \left\{
\begin{array}{ll}
 D_m(z), & \mbox{if} ~ m \equiv 0 ~(\bmod~ p);\\
0,    & otherwise.
\end{array}
\right.
\end{eqnarray*}
\end{Cor}

\section{Proof of Theorem~\ref{compositions}}

In this section we prove Theorem~\ref{compositions}. Obviously the
result holds trivially for $m=1$ because $S_1(0)=0$ and $S_1(z)=1$
for any $z\in \fq^*$. Moreover, when $m=2$, it is easy to see that
$S_2(0)=q-1$ and $S_2(z)=q-2$ where $z\in \fq^*$. Indeed,
$a+(p-1)a=0$ for any $a\in \fq^*$, but $a+x=z$  where $z\neq 0$ has
a nonzero solution for each $a\in\fq^*-\{z\}$.

Assume now $m \geq 3$ and   $z=x_1+x_2+\cdots +x_m$  for a fixed
$z\in \fq$. We consider $x_1+\cdots +x_{m-2}$ in the following two
cases:

(i) $x_1+\cdots +x_{m-2}=z$. In this case, there  are
$(q-1)S_{m-2}(z)$  solutions to $z=x_1+x_2+\cdots +x_m$ with all
$x_i$'s not equal to $0$  because  we can always find $x_m=-x_{m-1}$
such that $x_{m-1}+x_m=0$ for any choice of $x_{m-1}\in \fq^*$.

(ii) $x_1+\cdots +x_{m-2}\neq z$. As we can choose $x_1,\ldots,
 x_{m-2}$ in $(q-1)^{m-2}$ ways there are $(q-1)^{m-2}-S_{m-2}(z)$
 such ordered tuples. But $x_{m-1}\in \fq^*$ can not be equal to 
 $z-(x_1+x_2+\cdots +x_{m-2})\neq 0$  because this would imply $x_m=0$.
 Therefore we have only $q-2$ choices for $x_{m-1}$ and $x_m$ is
 uniquely determined by $x_m=z-(x_1+\cdots +x_{m-1})$. Therefore
 there are $\left((q-1)^{m-2}-S_{m-2}(z)\right)(q-2)$ solutions to $z=x_1+x_2+\cdots +x_m$ with all $x_i$'s not equal to $0$   in this case.

 \par Now summing up these numbers we obtain
$$S_m(z)= (q-1) S_{m-2}(z)+ \left((q-1)^{m-2}-S_{m-2}(z)\right)(q-2)=(q-1)^{m-2}(q-2)+S_{m-2}(z).$$

Therefore,  when $m=2k$,  we have
\begin{eqnarray*}
S_m(0)&=&(q-1)^{m-2}(q-2)+S_{m-2}(0)\\
      &=& (q-1)^{m-2}(q-2)+(q-1)^{m-4}(q-2)+S_{m-4}(0) \\
        &=& \ldots \\
      &=& (q-2)[(q-1)^{m-2}+(q-1)^{m-4}+\cdots +(q-1)^2]+S_2(0) \\
        &=& (q-2)\frac{(q-1)^m-1}{(q-1)^2-1}-(q-2)+(q-1) \\
        &=&(q-2)\frac{(q-1)^m-1}{(q-1)^2-1}+1\\
%%      &=&(q-2)\frac{(q-1)^m-1}{q(q-2)}+1 \\
        &=&\frac{(q-1)^m-1+q}{q}.
\end{eqnarray*}
Similarly for $z\neq 0$ and $m=2k$, we have
$$S_m(z)=\frac{(q-1)^m-1}{q}.$$
Moreover, if $m=2k+1$ then we have
\begin{eqnarray*}
S_m(0)&=&(q-1)^{m-2}(q-2)+S_{m-2}(z)\\
&=&(q-1)^{m-2}(q-2)+(q-1)^{m-4}(q-2)+S_{m-4}(z)\\
&=& \ldots \\
&=&(q-2)[(q-1)^{m-2}+(q-1)^{m-4}+\cdots +(q-1)]+S_1(0)\\
&=&
(q-2)(q-1)\frac{(q-1)^{m-1}-1}{(q-1)^2-1}=\frac{(q-1)^m-(q-1)}{q},
\end{eqnarray*}
and similarly
$$S_m(z)=(q-2)(q-1)\frac{(q-1)^{m-1}-1}{(q-1)^2-1}+1=\frac{(q-1)^m+1}{q}$$ for $z\in
\fq^*$. This completes the proof of Theorem~\ref{compositions}.

\begin{Cor}
The number of weak $m$-composition of $z\in \fq$ is
$$\sum_{k=0}^m \binom{m}{m-k}S_k(z)=q^{m-1}$$
\end{Cor}
\begin{proof}For each composition with $m-k$ nonzero elements, there are $\binom{m}{m-k}$ subsets of variables  $x_i$ that takes value zero and for the rest of variables we can have a composition of $z$ with $k$ parts. Thus there are $\binom{m}{m-k}S_k(z)$ solutions of the diagonal equation with $m-k$ variables equals to
the zero. Summing up these numbers we complete the proof.
\end{proof}

\begin{Cor} The number of solutions, none of $x_i$ is zero for $i=1, \ldots, n$, to the diagonal equation
$$x_1^{u_1}+x_2^{u_2}+\cdots +x_n^{u_n}=z$$ where $z\in \fq $ and
$u_1,u_2,\ldots,u_n$ are relatively prime to the $q-1$ is given by
$S_n(z)$. If $x_i$ is allowed to be zero, then the number of
solutions is $q^{m-1}$.

If  all but one exponent, say for example $u_n$, out of $u_1,
\ldots, u_n$ are relatively prime to $q-1$ and  $d= \gcd(q-1,u_n)
>1$,  then the number of the solutions of the
corresponding diagonal equation, where all $x_i\neq 0$, $i=1,\ldots, n$,  
is $dS_{n-1}(0)+(q-1-d)S_{n-1}(1)$ if there exists $u\in \fq^*$
such that $s=u^d$ and  the number is  $(q-1)S_{n-1}(1)$ otherwise.

\end{Cor}
\begin{proof} Assume first that all of  $u_1,u_2,\ldots, u_n$ are
relatively prime to the $q-1$. Because each mapping $x\mapsto
x^{u_i}$ is a bijection and thus  the number of solutions of the
diagonal equation above is equal to the number of the compositions
of $z$ into $n$ parts. \\
If $\gcd(u_n,q-1)=d$ and $\gcd(u_i, q-1) =1$ for $i=1, \ldots, n-1$,
then we have a  diagonal equation of the form
$$x_1^{u_1}+\cdots +x_{n-1}^{u_{n-1}}=z-x_n^{u_n}.$$
If $z\neq w^d$ for all $w\in \fq^*$ then for all $x_n\in \fq^*$
$z-x_n^{u_n}\neq 0$ and thus the number of solutions to the above
equation is $\displaystyle\sum_{x_n\in
\fq^*}S_{n-1}(z-x_n^{u_n})=(q-1)S_{n-1}(1)$; otherwise, 
for $d$ values of $x_n$ we have $z-x_n^{u_n}=0$. Hence the
number of solutions in this case is $\displaystyle\sum_{x_n\in
\fq^*}S_{n-1}(z-x_n^{u_n})=dS_{n-1}(0)+(q-1-d)S_{n-1}(1)$.
\end{proof}

\section{Proof of Theorem~\ref{main}}

Let $\ell = q-m$.  The assumption $\frac{q}{2} \leq \ell < q-3$
implies that $4 \leq m \leq \frac{q}{2}$. As in \cite{MuWa}, we
denote by $\mathcal{T}$ the family of all subsets of $\fq $ of
cardinality $m$, i.e.,
$$\mathcal{T}=\{T \mid T\subseteq \fq, |T|=m\}.$$
Denote by $\mathcal{M}$ the family of all multisets $M$ of order $q$
containing $0$ with the highest multiplicity $\ell = q- m$ and the
sum of elements in $M$ is equal to $0$, i.e.,
$$
\mathcal{M}=\{M \mid 0\in M, \text{ multiplicity}(0)=q-m, \sum_{b\in
M}b=0\}.
$$

We note that the polynomial with the least degree $q-m$ such that it sends  $q-m$ values to $0$ can be represented by
\begin{equation} \label{poly}
f_{(\lambda ,T)}(x)=\lambda \prod_{s\in \fq\setminus T}(x-s),
\end{equation}
which uniquely determines a mapping
\begin{equation}\label{mapping}
\mathcal{F}:\fq^* \times \mathcal{T}\rightarrow \mathcal{M},
\end{equation}
defined by 
$$(\lambda ,T)\mapsto \text{range}(f_{\lambda, T}(x)).$$

In Lemma 2  \cite{gacsetal} we found an upper bound for the number
$|\text{range}(\mathcal{F})|$ of the images of the polynomial with the
least degree $q-m$ such that it  sends  $q-m$ values  to $0$, when $m <
p$. Using this upper bound, we proved  that, for every $m$ with $3 <
m \leq \min\{p-1, q/2\}$,  there exists a multiset $M$ with
$\sum_{b\in M} b =0$ and the highest multiplicity $q-m$ achieved at
$0\in M$ such that every polynomial over  $\fq$ with
the prescribed range $M$ has degree greater than $q-m$ (Theorem ~1,
\cite{MuWa}). This result disproved  Conjecture 5.1 in
\cite{gacsetal}.  In this section, we drop the restriction of $m <p$
and then use the formula obtained in Theorem~\ref{partitions} to
prove Theorem~\ref{main}, which generalizes Theorem~1 in
\cite{MuWa}. First of all, we prove the following result.

\begin{Lem} Let $q$ be a prime power,  $m\leq \frac{q}{2}$ be a positive integer and $d = \gcd(q-1, m-1)$.
Let  $\mathcal{F}: \fq^*\times \mathcal{T} \rightarrow \mathcal{M}$
be defined as in Equation~(\ref{mapping}). Then
$$|\text{range}(\mathcal{F})|\leq \frac{(q-1)(q-2)\ldots
(q-m+1)}{m!}+\sum_{{i\mid d \atop i >1}
}\phi(i)\binom{\frac{q-1}{i}}{\frac{m-1}{i}}+\frac{\delta (q-1)}{q}\binom{q/p}{m/p},$$
where $\delta=1$ if $p\mid m$ and zero otherwise.
\end{Lem}
\begin{proof}
As in Lemma 2 of \cite{gacsetal} we consider the group $\mathcal{G}$
of all non-constant linear polynomials in $\fq[x]$ acting on the set
$\fq^*\times \mathcal{T}$ with action $\Phi
:(cx+b,(\lambda,T))\mapsto (c^{m-1}\lambda,cT+b)$. All the elements
of the same orbit in $\fq^*\times \mathcal{T}$ are all mapped to the
same range $M\in \mathcal{M}$. Thus we need to find the number $N$
of orbits under this group action.  Using the Burnside's Lemma, we
need to find the number of fixed points $|(\fq^*\times
\mathcal{T})_g|$ in $\fq^*\times \mathcal{T}$ under the action of
$g(x) = cx+b$.  As in Lemma 2 \cite{gacsetal},  for $g(x)=x$ there
are $(q-1)\binom{q}{m}$ elements fixed by $g(x)$. Moreover, if
$g(x)=cx+b$, $c\neq 1$ then elements are fixed by $g(x)$ only if $i
= ord(c)  \mid d = \gcd (q-1, m-1)$ and in this case we have
$|(\fq^*\times
\mathcal{T})_g|=(q-1)\binom{\frac{q-1}{i}}{\frac{m-1}{i}}$. Under
the assumption $m <p$ in Lemma 2 \cite{gacsetal}, we don't need to
consider $g(x)=x+b$, $b\neq 0$,  because it  has $p$-cycles of the
form $(x,x+b,\ldots ,x+(p-1)b)$ and has no fixed elements. However,
for arbitrary $m$, we must consider this case.  In fact, if $g(x) =
x+b$ fixes some subset $T$ of $\fq$ with $m$ elements then we must
have $p\mid m$ and  $T$ consists of $p$-cycles. In particular, there
are $\binom{\frac{q}{p}}{\frac{m}{p}}$ of such subsets $T$ fixed by
$g(x) =x+b$ for each $b\in \fq^*$. Varying $\lambda$ and $b$, we
therefore obtain $|(\fq^*\times \mathcal{T})_g|=
\delta (q-1)^2\binom{q/p}{m/p}$. Now using Burnside's Lemma we obtain
\begin{eqnarray*}
N &=&\frac{1}{|\mathcal{G}|}\sum_{g\in
\mathcal{G}}|(\fq^*\times\mathcal{T})_g| \\
&=& \frac{1}{q(q-1)}\Bigl((q-1)\binom{q}{m}+q(q-1)\sum_{i>0,i\mid
d}\phi(i)\binom{\frac{q-1}{i}}{\frac{m-1}{i}}+ \delta (q-1)^2\binom{q/p}{m/p}\Bigr)\\
&=&\frac{1}{q}\binom{q}{m}+\sum_{i>0,i\mid
d}\phi(i)\binom{\frac{q-1}{i}}{\frac{m-1}{i}}+\frac{\delta (q-1)}{q}\binom{q/p}{m/p}.
\end{eqnarray*}
\end{proof}

In order to prove Theorem~\ref{main}  it is clear that we only need
to show
\begin{equation}\label{Gacseq}
\frac{(q-1)(q-2)\ldots (q-m+1)}{m!}+\sum_{{i\mid d \atop i >1}
}\phi(i)\binom{\frac{q-1}{i}}{\frac{m-1}{i}}+\frac{\delta (q-1)}{q}\binom{q/p}{m/p}
<\tilde{P}_m(0).\end{equation}

By Theorem~\ref{partitions}, it is enough to show

\begin{equation}\label{Gacseq1}
\frac{(q-1)(q-2)\ldots (q-m+1)}{m!}+\sum_{{i\mid d \atop i >1}
}\phi(i)\binom{\frac{q-1}{i}}{\frac{m-1}{i}}+\frac{q-1}{q}\binom{q/p
- 1+j}{j}< \frac{1}{q}\binom{q+m-2}{m}.
\end{equation}
for $m = jp + 1$ and
\begin{equation}\label{Gacseq0}
\frac{(q-1)(q-2)\ldots (q-m+1)}{m!}+\sum_{{i\mid d \atop i >1}
}\phi(i)\binom{\frac{q-1}{i}}{\frac{m-1}{i}} <
\frac{1}{q}\binom{q+m-2}{m},
\end{equation}
for all other cases, because $\frac{q-1}{q}\binom{q/p}{m/p}  =
\frac{q-1}{q}\binom{q/p}{j} \leq \frac{q-1}{q}\binom{q/p - 1+j}{j}$
when $m=jp$ and $j\geq1$.

For the cases $m=4$ and $m=5$, because $q\geq 2m$,  we can check
directly that Inequality~(\ref{Gacseq0}) holds and thus
Inequality~(\ref{Gacseq}) holds.

\par We now show Inequalities~(\ref{Gacseq1}) and ~(\ref{Gacseq0}) hold for $m>5$ by using a combinatorial argument. Let $G=<a>$ be a cyclic group of
order $q-1$ with generator $a$.   Let $\mathcal{M^\prime}$ be the set of all multisets with
$m$ elements chosen from $G$. Then $|\mathcal{M^\prime}|=\binom{q-2+m}{m}$.
To estimate the left hand side of Inequalities~(\ref{Gacseq1}) and ~(\ref{Gacseq0})
 we count now the number
of multisets in some subsets of $\mathcal{M^\prime}$ defined as follows. These subsets of multisets of $m$ elements are defined from 
subsets of  $k$-subsets of $G$ when $k \leq m$.  First of all,  let
$\mathcal{M}_0$ be the set of all subsets of $G$ with $m$ elements.
So $\mathcal{M}_0 \subseteq \mathcal{M^\prime}$ and $|\mathcal{M}_0 |=
\binom{q-1}{m}$.

Let $\mathcal{A}$ be the set of all subsets of $G$ with $m-1$
elements. For each $A=\{a^{u_1},a^{u_2},\ldots ,a^{u_{m-1}}\} \in
\mathcal{A}$ where $0\leq u_1<u_2<\ldots <u_{m-1}<q-1$ we can find a
multiset $M=\{a^{u_1},a^{u_1},a^{u_2},a^{u_3},\ldots ,a^{u_{m-1}}\}$
corresponding to $A$ in the unique way.  We can use notation
$s^{(i)}$ to denote  an element $s$ in a multiset $M$ with
multiplicity $i$. Hence the above multiset $M$ can also be denoted by
$$M=\{ (a^{u_1})^{(2)},a^{u_2},a^{u_3},\ldots ,a^{u_{m-1}}\}.$$ The
set  of all these multisets $M$, denoted by $\mathcal{M}_1$,  has
$|\mathcal{A}|=\binom{q-1}{m-1}$ elements. Moreover
$\mathcal{M}_0\cap \mathcal{M}_1 =\emptyset $. Now let
$\mathcal{M}_{01}=\mathcal{M}_1\cup\mathcal{M}_1$. Then
$|\mathcal{M}_{01}|=\binom{q-1}{m}+\binom{q-1}{m-1}=\binom{q}{m}$.

For each $i$ satisfying $m-1>i \geq 2$ and $i\mid d$, we let
$S_i=<a^i>$ be a cyclic subgroup of $G$ with $\frac{q-1}{i}$
elements. From each set $\mathcal{C}_i$ of all subsets of $S_i$ with
$\frac{m-1}{i}$ elements, we can define two disjoint subclasses of
$\mathcal{M}$ containing multisets with $m$ elements in $G$
corresponding to $\mathcal{C}_i$.

First, let $B=\{a^{u_1i},a^{u_2i},\ldots, a^{u_{\frac{m-1}{i}}i}\}$
be a subset of $S_i$ where $0\leq u_1<u_2<\ldots <\frac{q-1}{i}$.
For each fixed $t$ such that $0\leq t<i$ and $\gcd(i,t)=1$, we can
construct a multiset corresponding to $B$ as follows:
$$M=\{(a^ta^{u_1i})^{(i)},(a^ta^{u_2i})^{(i)},\ldots
,(a^ta^{u_{\frac{m-1}{i}}i})^{(i)},a_m\}$$ where $a_m$ is
arbitrarily element in $G$. For each fixed $t$  this class of 
multisets formed from $\mathcal{C}_i$ is denote by
$\mathcal{M}_i^t$. Then
$|\mathcal{M}_i^t|=(q-1)\binom{\frac{q-1}{i}}{\frac{m-1}{i}}$.

Secondly, for  $B=\{a^{u_1i},a^{u_2i},\ldots,
a^{u_{\frac{m-1}{i}}i}\}\in \mathcal{C}_i$  and each fixed $t$, we
can construct another  multiset
$$\tilde{M}=\{ {\bf (a^{t+1}a^{u_1i})^{(i)}} ,(a^ta^{u_2i})^{(i)},\ldots
,(a^ta^{u_{\frac{m-1}{i}}i})^{(i)}, {\bf 1}\},$$ corresponding to $B$.  The
set of these multisets is denoted by $\tilde{\mathcal{M}}_i^t$. Then
$|\tilde{\mathcal{M}}_i^t|=\binom{\frac{q-1}{i}}{\frac{m-1}{i}}$.

Note that $i\leq \frac{m-1}{2}$ implies $\mathcal{M}_i^t\cap
\tilde{\mathcal{M}}_i^t=\emptyset$.  Hence we have
\[
|\mathcal{M}_i|= \left| {\displaystyle \bigcup_{1\leq t<i \atop
\gcd(i,t)=1} \mathcal{M}_i^t\cup \tilde{\mathcal{M}}_i^t } \right|
=\phi(i)
\left((q-1)\binom{\frac{q-1}{i}}{\frac{m-1}{i}}+\binom{\frac{q-1}{i}}{\frac{m-1}{i}}\right)
=q\phi(i)\binom{\frac{q-1}{i}}{\frac{m-1}{i}}.
\]

  Finally,  if $m-1\nmid q-1$ then we let $\mathcal{M}_m=\emptyset $.
  Otherwise, if $(m-1)\mid q-1$ then we let $\mathcal{M}_{m-1}^t$ contains all
the multisets of the form $M=\{(a^{t+j(m-1)})^{(m-1)},a_m\}$, for
$j=0,1,\ldots, \frac{q-1}{m-1}-1$, any positive integer $t < m-1$
with $\gcd(m-1,t)=1$, and any $a_m\in G$. Let
$\tilde{\mathcal{M}}_{m-1}^t$ contain all the multisets of the form $\{(a^{t+j(m-1)})^{(m-2)},(a^{m-1})^{(2)}\}$.
It is obvious that $a^{m-1}\neq a^{t+j(m-1)}$. By comparing the
multiplicities of two multisets we see that $\mathcal{M}_{m-1}^t\cap
\tilde{\mathcal{M}}_{m-1}^t=\emptyset$. Moreover,
\begin{eqnarray*}
|\mathcal{M}_{m-1}| &=& \left| \bigcup_{1\leq t<m-1 \atop
\gcd(m-1,t)=1}\mathcal{M}_{m-1}^t\cup \tilde{\mathcal{M}}_{m-1}^t \right|  \\
&=&\phi(m-1)\left((q-1)\binom{\frac{q-1}{m-1}}{1}+\binom{\frac{q-1}{m-1}}{1} \right)\\
&=& q \phi(m-1)\binom{\frac{q-1}{m-1}}{\frac{m-1}{m-1}}.
\end{eqnarray*}

\par Finally, if $m\neq jp+1$ for some $j\geq 1$ we let
$\mathcal{M}_m=\emptyset $. Otherwise, if $m=jp+1$ for some $j\geq
1$  we let $C = \{s_1,s_2,\ldots, s_{q/p}\}$ be a subset of $G$ with $q/p<q-1$ elements. For each
subset of $j$ elements from $C$ we
find a corresponding   multiset $M$ in $\mathcal{M}_m$ from
$\mathcal{M}$ in the following way
$$M=\{s_1^{(p)},s_2^{(p)},\ldots, s_j^{(p)},a_m\}$$
where $a_m$ is arbitrary chosen to be an element from $G$. Thus
there are $(q-1)\binom{q/p+j-1}{j}$ multisets in $\mathcal{M}_m$.
Obviously,   $\mathcal{M}_m$ is disjoint $\mathcal{M}_i$ where $i
\mid \gcd (m-1, q-1)$ because the multiplicity of at least one of
its element is $p\nmid q-1$. Indeed, it could possibly have common elements
only with $\mathcal{M}_{m-1}$ but in this case $m-1=jp\nmid q-1$ so
$\mathcal{M}_{m-1}=\emptyset$. Now
$|\mathcal{M}_m|=(q-1)\binom{q/p+j-1}{j}$.

 Define $\delta^\prime =0$ if $m\neq
jp+1$ for some $j$ and $\delta^\prime =1$ if $m=jp+1$. Then we obtain
\begin{eqnarray*}
|\mathcal{M}_{LHS}| &:=& \left| \mathcal{M}_{01}\bigcup
\left(\bigcup_{i>1 \atop i\mid
\gcd(m-1,q-1)}\mathcal{M}_i\right) \bigcup \mathcal{M}_m \right | \\
&=&\binom{q}{m}+q\sum_{{i\mid d \atop i >1}
}\phi(i)\binom{\frac{q-1}{i}}{\frac{m-1}{i}}+ \delta^\prime (q-1)\binom{q/p+(m-1)/p-1}{(m-1)/p}.
\end{eqnarray*}

We note that the multiset $\{1,1,1,a,a^2,\ldots ,a^{m-3}\}$ is not
included in the $\mathcal{M}_{LHS}$ and thus
$|\mathcal{M}_{LHS}|<|\mathcal{M^\prime}|$. Dividing both sides by $q$,  we have
\begin{equation} \label{upperbound}
\frac{1}{q} \binom{q}{m}+\sum_{{i\mid d \atop i >1}
}\phi(i)\binom{\frac{q-1}{i}}{\frac{m-1}{i}}+\frac{\delta^\prime (q-1)}{q}\binom{q/p+(m-1)/p-1}{(m-1)/p}
<\frac{1}{q} \binom{q+m-2}{m}.
\end{equation}
Hence both Inequalities~(\ref{Gacseq1}) and (\ref{Gacseq0}) are
satisfied. This completes the proof of Theorem~\ref{main}.


\begin{thebibliography}{10}

\bibitem{BEW} B. C. Berndt, R. J. Evans, K. S. Williams,  Gauss and Jacobi sums.
Canadian Mathematical Society Series of Monographs and Advanced
Texts. A Wiley-Interscience Publication. John Wiley \& Sons, Inc.,
New York, 1998.


\bibitem{gacsetal} A. G\'{a}cs, T. H\'{e}ger, Z. L. Nagy, D. P\'{a}lv\"{o}lgyi, Permutations,
hyperplanes and polynomials over finite fields, {\it Finite Field
Appl.}  \textbf{16} (2010), 301-314.

\bibitem{LiWan} J. Li and D. Wan, On the subset sum problem over finite
fields, {\it Finite Field Appl.}, \textbf{14} (2008), 911-929.

\bibitem{MuWa} A. Muratovi\'{c}-Ribi\'{c} and Q. Wang,
On a conjecture of polynomials with prescribed range, \emph{Finite
Field Appl. (2012)}, doi:10.1016/j.ffa.2012.02.004.

\bibitem{Stanley} R. P. Stanley, Enumerative Combinatorics, Vol I, Cambridge University Press, 1997.

\end{thebibliography}
\end{document}